\documentclass[11pt]{amsart}

\usepackage[english]{babel}
\usepackage{amssymb}

\usepackage{amsmath}
\usepackage{enumitem}
\usepackage{colonequals} 
\usepackage{comment}
\usepackage{graphicx}
\usepackage{xcolor}
\usepackage[colorlinks=true, allcolors=darkgray!95]{hyperref}

\newtheorem{thm}{Theorem}
\newtheorem{lem}[thm]{Lemma}
\newtheorem{prop}[thm]{Proposition}

\theoremstyle{remark}
\theoremstyle{definition}

\newtheorem{ex}[thm]{Example}
\newtheorem{rem}[thm]{Remark}

\newcommand{\N}{\mathbb{N}}
\newcommand{\Z}{\mathbb{Z}}
\newcommand{\F}{\mathbb{F}}
\renewcommand{\P}{\mathrm{P}}

\title[Hilbert's basis theorem for Poisson Ore extensions]{Hilbert's basis theorem for\\ Poisson Ore extensions}

\date{\today}

\author{Per B\"{a}ck}
\address[Per B\"{a}ck]{Department of Business and Mathematics, M\"{a}lardalen  University, SE-721 23  V\"{a}ster{\aa}s, Sweden}
\email{per.back@mdu.se}

\author{Patrik Lundstr\"{o}m}
\address[Patrik Lundstr\"{o}m]{Department of Engineering Science, University West, SE-461 86  Trollh\"{a}ttan, Sweden}
\email{patrik.lundstrom@hv.se}

\author{Johan \"{O}inert}
\address[Johan \"{O}inert]{Department of Mathematics and Natural Sciences, Blekinge Institute of Technology, SE-371 79 Karlskrona, Sweden and Department of Engineering, University of Sk\"{o}vde, SE-541 28 Sk\"{o}vde, Sweden}
\email{johan.oinert@bth.se}

\author{Johan Richter}
\address[Johan Richter]{Department of Mathematics and Natural Sciences, Blekinge Institute of Technology, SE-371 79 Karlskrona, Sweden}
\email{johan.richter@bth.se}

\begin{document}

\begin{abstract}
We prove an analogue of Hilbert's 
basis theorem for
Poisson Ore extensions and 
Poisson Laurent Ore extensions. 
We also obtain corresponding results for iterated Poisson Ore extensions and iterated Poisson Laurent Ore extensions associated to commuting Poisson-pairs.
Finally, we give examples of Poisson Ore extensions that are Poisson--Noetherian without being Noetherian as ordinary algebras.
\end{abstract}

\subjclass[2020]{Primary 17B63; Secondary 13E05, 16P40, 16S36}

\keywords{Poisson algebra, Poisson Ore extension, Poisson Laurent Ore extension, Poisson--Noetherian algebra, Hilbert's basis theorem, Poisson ideal}

\maketitle

\section{Introduction}

Poisson algebras provide an algebraic framework for structures arising in classical mechanics, deformation quantization, quantum groups, algebraic geometry, representation theory, and integrable systems \cite{LGPV2012}. 
Originating in the work of Poisson and later developed in both geometric and algebraic directions, the theory now forms a bridge between commutative algebra, Lie theory, and noncommutative geometry; see, e.g., \cite{BV1988,RS2023}
for more background on Poisson algebras.

Throughout, let $K$ be an associative
commutative unital ring. Recall that a \emph{Poisson algebra} is an associative commutative unital $K$-algebra $A$ equipped with a 
$K$-bilinear operation $\{\cdot,\cdot\}_A 
\colon A \times A \to A$, called a 
\emph{Poisson bracket},
defining a Lie algebra structure on $A$, that is,
\[
\{a,a \}_A = 0 \quad \mbox{and} \quad  
\{ a, \{ b,c \}_A \}_A = 
\{\{a,b \}_A, c \}_A +
\{b,\{ a, c \}_A \}_A,
\]
for all $a,b,c \in A$, and satisfying Leibniz's rule,
\[
\{a, b c \}_A = 
\{a, b\}_A c + b \{a, c\}_A,
\]
for all $a,b,c \in A$. By a \emph{derivation} of a $K$-algebra, we always mean a $K$-linear derivation. A derivation $\delta$ of $A$ is called a \emph{Poisson derivation} if
\[
\delta(\{a,b\}_A) = 
\{\delta(a),b\}_A + \{a,\delta(b)\}_A,
\]
for all $a,b \in A$. 
A Poisson derivation $\delta$ is said to be \emph{Hamiltonian} if there exists $a\in A$ such that 
$\delta(b) = \{ a,b \}_A$ for all $b \in A$.

Inspired by Ore's \cite{Ore1933} 
construction of noncommutative polynomial 
rings, now called \emph{Ore extensions}, 
Oh~\cite{Oh2006}
introduced 
Poisson polynomial rings,
now often referred to as 
\emph{Poisson Ore extensions},
in the following way. Let $A$
be a Poisson algebra, and let $A[x]$ be the ordinary polynomial algebra. Oh~\cite{Oh2006} proves the following result over fields; however, the same proof applies verbatim to Poisson algebras over arbitrary associative commutative unital rings. A Poisson bracket
$\{ \cdot , \cdot \}$ on
$A[x]$ with the property that it 
extends $\{ \cdot,\cdot \}_A$,
and satisfies for all $a \in A$,
$\{ a,x \} = 
\alpha(a)x + \delta(a),$
for some maps
$\alpha,\delta\colon A \to A$, exists
if and only if
$\alpha$ is a Poisson derivation
of $A$ and $\delta$ is a 
derivation of $A$ that satisfy, 
for all $a,b\in A$, 
\begin{equation}\label{eq:mixedidentity}
\delta(\{a,b\}_A) - 
\{\delta(a),b\}_A - 
\{a, \delta(b)\}_A 
= \delta(a)\alpha(b) -
\alpha(a)\delta(b).
\end{equation}
In that case, we say that 
$(\alpha,\delta)$ is
a \emph{Poisson-pair} on $A$.
Oh~\cite{Oh2006} also shows that such 
an extension is unique and that
\begin{equation}\label{eq:monomials}
\{ ax^i, bx^j \} 
=( \{a,b\}_A  +
j\alpha(a)b -
 ia\alpha(b)) x^{i+j} 
+ (j\delta(a)b -
ia\delta(b))x^{i+j-1}
\end{equation}
holds for all $a,b\in A$ and 
$i,j\in \N \colonequals \{ 0,1,2, \ldots \}$,
with the convention that the second summand is omitted when $i=j=0$.
The resulting structure is called a
\emph{Poisson Ore extension} and it
is denoted by $A[x;\alpha,\delta]_\P$.
This family of algebras includes many 
classical constructions; 
see e.g.~\cite[Section 2]{Oh2006}.

The Poisson bracket $\{\cdot,\cdot\}$ on $A[x;\alpha,\delta]_\P$ extends uniquely to a Poisson bracket $\{\cdot,\cdot\}$ 
on the Laurent polynomial algebra $A[x^{\pm}]$ (see Remark~\ref{rem:localise}). In this case, \eqref{eq:monomials} holds for all $a,b\in A$ and $i,j\in \mathbb Z$. We call the resulting Poisson algebra a \emph{Poisson Laurent Ore extension} 
and denote it by $A[x^{\pm};\alpha,\delta]_\P$.

A key result about polynomial rings, and a cornerstone of commutative algebra, is \emph{Hilbert’s basis theorem}, first established by Hilbert~\cite[Theorems I and II]{Hil1890}. It states that if a ring $R$ is Noetherian, then so is the polynomial ring $R[x]$; a first version for Ore extensions was established by Noether and Schmeidler~\cite[Satz III]{NS1920} (see e.g.~\cite[Theorem 2.6]{GW2004} for the general version).
Related work includes Varadarajan’s generalization
of Hilbert’s basis theorem for polynomial modules~\cite[Theorem A]{Var1982} and more recent nonassociative generalizations of Ore extensions~\cite{AB2025, BLOR2026, BR2023, BR2024}.
These results motivate the corresponding
question for Poisson Ore extensions and
Poisson Laurent Ore extensions. In this article, we prove a Hilbert's basis theorem for Poisson ideals
(see Theorem \ref{thm:main}). 

Let $A$ be a Poisson algebra. An ideal $I$ of $A$ is called a \emph{Poisson ideal} if $\{A,I\}_A\subseteq I$. Given $S\subseteq A$, a Poisson ideal $I$ is \emph{generated by $S$} if $I$ is the intersection of all Poisson ideals of $A$ containing $S$. We say that $A$ is \emph{Poisson--Noetherian} if every Poisson ideal of $A$ is finitely generated as a Poisson ideal; equivalently, if every ascending chain of Poisson ideals of $A$ stabilizes, or equivalently, if every nonempty collection of Poisson ideals of $A$ contains a maximal element with respect to inclusion.

\begin{thm}\label{thm:main}
Let $A$ be a Poisson algebra that is Poisson--Noetherian and 
possesses a Poisson-pair $(\alpha,\delta)$. Then $A[x ; \alpha,\delta]_\P$
and $A[x^{\pm} ; \alpha,\delta]_\P$
are Poisson--Noetherian.
\end{thm}

\section{Hilbert's basis theorem}

In this section, we prove Theorem~\ref{thm:main}. 
We then extend the result to iterated Poisson Ore extensions 
and iterated Poisson Laurent Ore extensions associated to 
commuting Poisson-pairs (see Theorem~\ref{thm:Iter-Laurent}). 
We end this section with examples illustrating the distinction 
between Poisson--Noetherianity and ordinary Noetherianity (see Examples~\ref{ex:example-Poisson-simplicity}, \ref{ex:Poisson-Noetherian}, \ref{ex:PoissonNoetherianNotNoetherian}, and \ref{ex:IteratedNotNoetherian}).

\begin{rem}\label{rem:localise}
Let $C$ be a Poisson algebra, and let 
$S\subseteq C$ be a multiplicatively closed set of regular elements.
A direct verification shows that the Poisson bracket on $C$ extends to $S^{-1}C$ by the quotient rule
\begin{displaymath}
	\left\{\frac{f}{s}, \frac{g}{t} \right\}
	\colonequals
	\frac{st \{f,g\}_C - ft \{s,g\}_C - gs \{f,t\}_C + fg \{s,t\}_C }{s^2t^2}
\end{displaymath}
for $f,g\in C$ and $s,t\in S$.
It follows by 
Leibniz's rule and the identity
$\{s^{-1},g\} = -s^{-2} \{s,g\}_C$
that the extended bracket is unique.

Since $x$ is regular in the ordinary polynomial algebra $A[x]$,
applying this construction to $C\colonequals A[x;\alpha,\delta]_\P$
and $S\colonequals\{1,x,x^2,\ldots\}$
yields a Poisson bracket on the Laurent polynomial algebra $A[x^\pm]$.
The resulting Poisson algebra is called a \emph{Poisson Laurent Ore extension} and is denoted by
$A[x^\pm;\alpha,\delta]_\P$. Note that \eqref{eq:monomials} holds for all $a,b \in A$ and $i,j \in \Z$.
\end{rem}

\begin{lem}\label{lem:localization-Noetherian}
Let $C$ be a Poisson algebra, and let
$S\subseteq C$ be a multiplicatively closed set of regular
elements. If $C$ is Poisson--Noetherian, then so is $S^{-1}C$.
\end{lem}

\begin{proof}
Let $I$ be an ideal of $S^{-1}C$. Then
$
I=(S^{-1}C)(I\cap C).$
Indeed, if $f/s\in I$, then $f/1=(s/1)(f/s)\in I$, so $f\in I\cap C$. Hence
$f/s=(1/s)(f/1)\in (S^{-1}C)(I\cap C)$, so $I\subseteq (S^{-1}C)(I\cap C)$. The reverse inclusion is immediate.

Let $I_1\subseteq I_2\subseteq I_3\subseteq\cdots$ be an ascending chain of Poisson ideals of $S^{-1}C$. Then $I_1\cap C\subseteq I_2\cap C\subseteq I_3\cap C\subseteq\cdots$ is an ascending chain of Poisson ideals of $C$. Suppose that $C$ is Poisson--Noetherian. Then there is a positive integer $k$ such that $I_k\cap C=I_n\cap C$ for every $n\geq k$. Therefore, $I_k=(S^{-1}C)(I_k\cap C)=(S^{-1}C)(I_n\cap C)=I_n$ for every $n\geq k$. Hence $S^{-1}C$ is
Poisson--Noetherian.
\end{proof}

\subsubsection*{Proof of 
Theorem~\ref{thm:main}.}
Let $(\alpha,\delta)$ be a 
Poisson-pair on a Poisson--Noetherian
Poisson algebra $A$.
We begin by showing that $B \colonequals A[x ; \alpha,\delta]_\P$
is Poisson--Noetherian. Seeking a contradiction, suppose that 
$L$ is a Poisson ideal of $B$ that is not finitely generated.

Define Poisson ideals 
$(L_i)_{i \in \N}$ of $B$ and polynomials
$(f_i)_{i \in \N}$ in $B$ 
recursively by first putting
$L_0 \colonequals \{ 0 \}$. 
Next, for each $i \in \N$,
take $f_i \in L \setminus L_i$ of minimal 
degree and let $L_{i+1}$ be the 
Poisson ideal of $B$ generated 
by the set $\{ f_0,\ldots,f_i \}$.

For each $i \in \N$, let $m_i \in A$
denote the leading coefficient of 
$f_i$. Let $M$ denote the Poisson ideal
of $A$ generated by $\{ m_i \}_{i \in \N}$.
Since $A$ is Poisson--Noetherian, 
there is $N \in \N_+$
such that $M$ is generated, as a Poisson 
ideal of $A$, by 
$\{ m_i \}_{i=0}^{N-1}$.

Put $D \colonequals \deg(f_N)$.
The sequence $(\deg(f_i))_{i \in \N}$ is non-decreasing, 
so for each $i \in \{0,\ldots,N-1 \}$, we may put
$g_i \colonequals 
x^{D-\deg(f_i)} f_i$.
Then, for each $i \in \{0,\ldots,N-1 \}$,
$g_i$ has leading coefficient $m_i$,
$\deg(g_i) = D$, and $g_i \in L_N$.

We claim that $L_N$ contains a polynomial
$g$ with leading coefficient $m_N$ and  
$\deg(g) = D$. 
Let us assume, for a moment,
that the claim holds. Since $f_N \in 
L \setminus L_N$ and $g \in L_N$, it follows
that $f_N - g \in L \setminus L_N$.
But $\deg(f_N - g) < D = \deg(f_N)$, which
contradicts the choice of $f_N$.

Now we show the claim. 
Let $P$ denote the set of all $m \in M$ such that
either $m = 0$ or $m$ is the leading coefficient
of a polynomial in $L_N$ of degree $D$.
Then $P$ is a $K$-submodule of $M$. We claim that
$P$ is a Poisson ideal of $A$.
First, $AP \subseteq P$. Indeed, take $a \in A$ and
$m \in P$. If $m=0$, then $am=0\in P$. Otherwise,
there is a polynomial $h\in L_N$ of degree $D$ with
leading coefficient $m$. Then $ah\in L_N$, and either
$am=0$ or $ah$ has degree $D$ and leading coefficient
$am$. Hence $am\in P$.

It remains to show that $\{A,P\}_A
\subseteq P$. Take
$a\in A$ and $m\in P$. If $m=0$, then $\{a,m\}_A=0\in P$.
Otherwise, choose $h\in L_N$ of degree $D$ with leading
coefficient $m$. Put
$q \colonequals \{a,h\} - D\alpha(a)h$.
Since $L_N$ is a Poisson ideal of $B$, we have $q\in L_N$.
Either $\{a,m\}_A = 0$, or $q$ has degree $D$ and, by \eqref{eq:monomials}, leading coefficient
$\{a,m\}_A$. 
In both cases, $\{a,m\}_A\in P$.

Thus $P$ is a Poisson ideal of $A$. Moreover,
$M$ is the Poisson ideal of $A$ generated by the set $\{m_0,\ldots,m_{N-1}\}\subseteq P$. Hence,
$M\subseteq P$. It follows that $m_N\in P$, which proves the claim.

Finally, $A[x^{\pm};\alpha,\delta]_\P$ is the localization
of $B$ at the multiplicative set $\{1,x,x^2,\ldots\}$. It is
therefore Poisson--Noetherian by
Lemma~\ref{lem:localization-Noetherian}.
\qed \\

We now extend Theorem~\ref{thm:main} to iterated Poisson
Ore extensions and iterated Poisson Laurent Ore extensions. We first record the following elementary
two-variable construction.

\begin{lem}\label{lem:extend-commuting-pair}
Let $A$ be a Poisson algebra, and let
$(\alpha,\delta)$ and $(\beta,\epsilon)$ be Poisson-pairs
on $A$. Suppose that each of the maps $\beta$ and $\epsilon$
commutes with each of the maps $\alpha$ and $\delta$.
Let $B \colonequals A[x;\alpha,\delta]_\P$.
Extend $\beta$ and $\epsilon$ to $B$ 
by setting
\[
\beta(ax^i)\colonequals \beta(a)x^i,
\qquad
\epsilon(ax^i)\colonequals \epsilon(a)x^i
\]
for all $a\in A$ and $i\in\N$, and then extend
$K$-linearly. Then $(\beta,\epsilon)$ is a Poisson-pair on
$B$. In particular, the Poisson Ore extension
$B[y;\beta,\epsilon]_\P$
is well defined and satisfies
$\{x,y\}=0$,
$\{a,y\}=\beta(a)y+\epsilon(a)$
for all $a\in A$.
\end{lem}

\begin{proof}
It is clear that the extended maps $\beta$ and $\epsilon$
are derivations of the ordinary polynomial algebra $A[x]$.
We show that they form a Poisson-pair on $B$.
Let $a,b\in A$ and $i,j\in\N$. By \eqref{eq:monomials},
\[
\{ax^i,bx^j\}
=
\bigl(\{a,b\}_A+j\alpha(a)b-ia\alpha(b)\bigr)x^{i+j}
+
\bigl(j\delta(a)b-ia\delta(b)\bigr)x^{i+j-1}.
\]
Using that $\beta$ is a Poisson derivation of $A$ and that
$\beta$ commutes with both $\alpha$ and $\delta$, a direct
calculation gives
\[
\beta(\{ax^i,bx^j\})
=
\{\beta(ax^i),bx^j\}
+
\{ax^i,\beta(bx^j)\}.
\]
Thus $\beta$ is a Poisson derivation of $B$.

Similarly, using that $(\beta,\epsilon)$ is a Poisson-pair
on $A$ and that $\epsilon$ commutes with both $\alpha$ and
$\delta$, the same formula gives
\[
\begin{aligned}
\epsilon(\{ax^i,bx^j\})
={}&
\{\epsilon(ax^i),bx^j\}
+
\{ax^i,\epsilon(bx^j)\}  \\
&-
\beta(ax^i)\epsilon(bx^j)
+
\epsilon(ax^i)\beta(bx^j).
\end{aligned}
\]
Hence $(\beta,\epsilon)$ is a Poisson-pair on $B$.

Finally, since $\beta(x)=0$ and $\epsilon(x)=0$, the defining
relation in $B[y;\beta,\epsilon]_\P$ gives
$\{x,y\}=\beta(x)y+\epsilon(x)=0$.
The relation
$\{a,y\}=\beta(a)y+\epsilon(a)$
for $a\in A$ follows directly from the definition of the
Poisson Ore extension.
\end{proof}

\begin{prop}\label{prop:iterated-Poisson-Ore-construction}
Let $A$ be a Poisson algebra, and let
$(\alpha_1,\delta_1),\ldots,(\alpha_n,\delta_n)$
be Poisson-pairs on $A$. Suppose that, whenever $i\neq j$,
each of the maps $\alpha_i$ and $\delta_i$ commutes with
each of the maps $\alpha_j$ and $\delta_j$.
Then the polynomial algebra
$A[x_1,\ldots,x_n]$
has a unique Poisson algebra structure extending the Poisson
bracket on $A$ and satisfying
$\{a,x_i\}=\alpha_i(a)x_i+\delta_i(a)$
for all $a\in A$ and $i\in\{1,\ldots,n\}$, and
$\{x_i,x_j\}=0$
for all $i,j\in\{1,\ldots,n\}$. We denote this Poisson
algebra by
$A[x_1;\alpha_1,\delta_1]_\P
\cdots
[x_n;\alpha_n,\delta_n]_\P$.
Moreover, its Poisson bracket extends uniquely to the
Laurent polynomial algebra
$A[x_1^{\pm},\ldots,x_n^{\pm}]$.
The resulting Poisson algebra
is denoted by
$A[x_1^{\pm};\alpha_1,\delta_1]_\P
\cdots
[x_n^{\pm};\alpha_n,\delta_n]_\P$.
\end{prop}

\begin{proof}
We construct the polynomial Poisson algebra by induction on
$n$. The case $n=1$ is precisely the construction of a
Poisson Ore extension.

Assume that the result holds for
$n-1$ variables, and put
\[
B_{n-1}
\colonequals
A[x_1;\alpha_1,\delta_1]_\P
\cdots
[x_{n-1};\alpha_{n-1},\delta_{n-1}]_\P.
\]
By repeated application of
Lemma~\ref{lem:extend-commuting-pair}, the maps
$\alpha_n$ and $\delta_n$ extend to a Poisson-pair on
$B_{n-1}$ by acting on the coefficients in $A$ and by
satisfying
$\alpha_n(x_i)=0$,
$\delta_n(x_i)=0$
for $i \in \{1,\ldots,n-1\}$. 
Hence the Poisson Ore extension
$B_n\colonequals B_{n-1}[x_n;\alpha_n,\delta_n]_\P$
is well defined. By construction,
$\{a,x_n\}=\alpha_n(a)x_n+\delta_n(a)$
for all $a\in A$, and
$\{x_i,x_n\}=0$ for $i \in \{1,\ldots,n-1\}$. Uniqueness of the Poisson algebra structure follows because a Poisson bracket is a biderivation and is therefore determined by its values on $A$ and on the variables $x_1,\ldots,x_n$.
This proves the polynomial case.

The multiplicative set $\{x_1^{m_1}\cdots x_n^{m_n}\mid m_1,\ldots,m_n\in\N\}$
consists of regular elements of the polynomial algebra $A[x_1,\ldots,x_n]$.
Hence, by Remark~\ref{rem:localise},
the Poisson bracket on $A[x_1,\ldots,x_n]$ extends uniquely to a Poisson bracket on the localization
$A[x_1^\pm,\ldots,x_n^\pm]$.\qedhere
\end{proof}

\begin{thm}\label{thm:Iter-Laurent}
Let $A$ be a Poisson algebra that is Poisson--Noetherian, and suppose that
$(\alpha_1,\delta_1),\ldots,(\alpha_n,\delta_n)$
are Poisson-pairs on $A$. Suppose that, whenever $i\neq j$,
each of the maps $\alpha_i$ and $\delta_i$ commutes with
each of the maps $\alpha_j$ and $\delta_j$.
Then \[ A[x_1;\alpha_1,\delta_1]_\P
\cdots
[x_n;\alpha_n,\delta_n]_\P \]
and
\[ A[x_1^{\pm};\alpha_1,\delta_1]_\P
\cdots
[x_n^{\pm};\alpha_n,\delta_n]_\P \]
are Poisson--Noetherian.
\end{thm}

\begin{proof}
The Poisson Ore extension case follows by induction on $n$ from
Theorem~\ref{thm:main}. 
The case $n=1$ is exactly
Theorem~\ref{thm:main}. If the result holds for $n-1$
variables, then
$B_{n-1}
\colonequals
A[x_1;\alpha_1,\delta_1]_\P
\cdots
[x_{n-1};\alpha_{n-1},\delta_{n-1}]_\P$
is Poisson--Noetherian. By
Lemma~\ref{lem:extend-commuting-pair}, the pair
$(\alpha_n,\delta_n)$ extends to a Poisson-pair on
$B_{n-1}$. Hence Theorem~\ref{thm:main} shows that
$B_{n-1}[x_n;\alpha_n,\delta_n]_\P$
is Poisson--Noetherian.

Denote by $S$ the multiplicative set $\{x_1^{m_1}\cdots x_n^{m_n}\mid m_1,\ldots,m_n\in\N\}$. The Poisson algebra $A[x_1^\pm;\alpha_1,\delta_1]_\P\cdots [x_n^\pm;\alpha_n,\delta_n]_\P$ is then the localization of $A[x_1;\alpha_1,\delta_1]_\P\cdots [x_n;\alpha_n,\delta_n]_\P$ at $S$. It is Poisson--Noetherian by  Lemma~\ref{lem:localization-Noetherian}.
\end{proof}

We now consider some examples of
Poisson Ore extensions.
Recall that a Poisson algebra 
$A$ is called \emph{Poisson-simple} 
if its only Poisson ideals are 
$\{0\}$ and $A$.
In the first example, 
we show that there exist non-Noetherian
algebras that are Poisson-simple and 
hence Poisson--Noetherian.

\begin{ex}\label{ex:example-Poisson-simplicity}
Let $\F$ be a field of characteristic zero and put
\[
A \colonequals \F[x_1,x_2,x_3,\ldots].
\]
Then $A$ is not Noetherian. Indeed, for each $n\geq 1$,
let $I_n$ be the ideal of $A$ generated by
$x_1,\ldots,x_n$. Then
$I_1 \subsetneq I_2 \subsetneq I_3 \subsetneq \cdots$.
For $f,g\in A$, define
\[
\{f,g\}_A
\colonequals
\sum_{i = 1}^{\infty}
\left(
\frac{\partial f}{\partial x_i}
\frac{\partial g}{\partial x_{i+1}}
-
\frac{\partial f}{\partial x_{i+1}}
\frac{\partial g}{\partial x_i}
\right).
\]
This sum is finite, since each 
polynomial in $A$ involves
only finitely many variables. 
It is easy to check that 
$\{\cdot,\cdot\}_A$ defines a 
Poisson bracket on $A$.

We claim that $A$ is Poisson-simple. For $p\in A$, define
$d(p)\in \N$ as follows. 
If $p\in \F$, set $d(p)=0$. 
If $p\notin \F$, let $d(p)$ be the least 
positive integer such that
$p \in \F[x_1,\ldots,x_{d(p)}]$.
Let $I$ be a nonzero Poisson ideal 
of $A$. Choose
$p\in I\setminus\{0\}$ such that 
$d(p)$ is minimal. We show
that $d(p)=0$. Seeking a contradiction,
suppose that $d(p)>0$, and
write $d=d(p)$. Let $r \colonequals \deg_{x_d}(p)$.
Then $r\geq 1$. Moreover,
$\{x_{d+1},p\}_A
= -\frac{\partial p}{\partial x_d}$.
Since $I$ is a Poisson ideal, repeated application of the
Hamiltonian derivation 
$\{ x_{d+1},\cdot \}_A$ gives
\[
\{x_{d+1},\{x_{d+1},\cdots\{x_{d+1},p\}_A\cdots\}_A\}_A
=
(-1)^r \frac{\partial^r p}{\partial x_d^r}
\in I,
\]
where $x_{d+1}$ occurs $r$ times. Since $\F$ has characteristic zero,
this element is nonzero. It belongs to
$\F[x_1,\ldots,x_{d-1}]$, and hence has $d$-value strictly smaller than $d(p)$. This contradicts the choice of $p$.
Therefore $d(p)=0$, so $p\in \F\cap I$ is a nonzero scalar.
Thus $1\in I$, and hence $I=A$. 
\end{ex}

Next, we use Example~\ref{ex:example-Poisson-simplicity}
to construct a Poisson Ore extension which is
Poisson--Noetherian but not Noetherian as an ordinary
algebra. 

\begin{ex}\label{ex:Poisson-Noetherian}
Let $A$ be the Poisson algebra from
Example~\ref{ex:example-Poisson-simplicity}. Set
$\alpha \colonequals \frac{\partial}{\partial x_1}$.
We first show that $\alpha$ is a Poisson derivation of $A$.
Since $A$ is generated as an algebra by the variables
$x_1,x_2,x_3,\ldots$, it is enough to check the identity
\[
\alpha(\{x_i,x_j\}_A)
=
\{\alpha(x_i),x_j\}_A+\{x_i,\alpha(x_j)\}_A
\]
for all $i,j\geq 1$. But each bracket $\{x_i,x_j\}_A$ is
a scalar, and each $\alpha(x_i)$ is also a scalar. Hence
both sides are zero, and $\alpha$ is a Poisson derivation.
We also note that $\alpha$ is not Hamiltonian. Suppose, to
the contrary, that there exists $a\in A$ such that $\alpha=\{a, \cdot \}_A$.
Since $a$ involves only finitely many variables, and since
$\alpha\neq 0$, there is a largest integer $N$ such that
$\frac{\partial a}{\partial x_N}\neq 0$.
Then
\[
0=\alpha(x_{N+1})
=\{a,x_{N+1}\}_A
=
\frac{\partial a}{\partial x_N}
-
\frac{\partial a}{\partial x_{N+2}}
=
\frac{\partial a}{\partial x_N},
\]
which is a contradiction. Thus $\alpha$ is not Hamiltonian.

Consider the Poisson Ore extension
$B\colonequals A[y;\alpha,0]_\P$.
Thus $B=A[y]$ as an ordinary algebra, 
and its Poisson
bracket is defined by
$\{a,y\}=\alpha(a)y$
for all $a\in A$.
By Theorem~\ref{thm:main}, 
the Poisson algebra $B$ is
Poisson--Noetherian, since $A$ is Poisson-simple and hence
Poisson--Noetherian. However, $B$ is not Noetherian as an ordinary algebra.
Indeed, $B/(y)\cong A$ as ordinary algebras, and $A$ is not
Noetherian. Hence $B$ is not Noetherian.

Finally, $B$ is not Poisson-simple. For each $n\geq 1$, the
ordinary ideal $y^nB$ is a proper Poisson ideal of $B$.
Indeed, it is clearly a proper ideal of $B$, and for
$a,b\in A$ and $i,j\in \N$, the formula
\[
\{ay^i,by^j\}
=
\bigl(\{a,b\}_A+j\alpha(a)b-ia\alpha(b)\bigr)y^{i+j}
\]
shows that the bracket of an arbitrary element of $B$ with
an element of $y^nB$ again belongs to $y^nB$. Hence
$y^nB$ is a proper Poisson ideal of $B$. In particular,
$B$ is not Poisson-simple.
\end{ex}

The preceding example used a 
Poisson-pair of the form
$(\alpha,0)$. The next example 
considers the opposite
case, namely Poisson-pairs of the 
form $(0,\delta)$.

\begin{ex}\label{ex:PoissonNoetherianNotNoetherian}
Let $A$ be the Poisson algebra from
Example~\ref{ex:example-Poisson-simplicity}, and let
$\delta$ be any Poisson derivation of $A$. Then
$(0,\delta)$ is a Poisson-pair on $A$. Indeed, since
$\delta$ is a Poisson derivation, we have
$\delta(\{a,b\}_A) =
\{\delta(a),b\}_A+\{a,\delta(b)\}_A$
for all $a,b\in A$, which is precisely
\eqref{eq:mixedidentity} with $\alpha=0$.
Hence Theorem~\ref{thm:main} implies that
$A[x;0,\delta]_\P$
is Poisson--Noetherian. 
As an ordinary algebra,
$A[x;0,\delta]_\P$ is equal to the polynomial algebra
$A[x]$. In particular, it is not Noetherian, since
$A[x;0,\delta]_\P/(x)\cong A$
as ordinary algebras, and $A$ is not Noetherian.
\end{ex}

In our final example, we consider an
iterated version of
Example~\ref{ex:Poisson-Noetherian}.

\begin{ex}\label{ex:IteratedNotNoetherian}
Let $A$ be the Poisson algebra from
Example~\ref{ex:example-Poisson-simplicity}. For a fixed
$n\geq 1$, set
$\alpha_i \colonequals \frac{\partial}{\partial x_i}$ for 
$i\in\{1,\ldots,n\}$.
As in Example~\ref{ex:Poisson-Noetherian}, each $\alpha_i$
is a Poisson derivation of $A$. Moreover, the derivations
$\alpha_1,\ldots,\alpha_n$ commute pairwise. Hence
$(\alpha_1,0),\ldots,(\alpha_n,0)$
are commuting Poisson-pairs on $A$. Therefore, by
Theorem~\ref{thm:Iter-Laurent}, $A[y_1;\alpha_1,0]_\P
\cdots
[y_n;\alpha_n,0]_\P$
and 
$A[y_1^{\pm};\alpha_1,0]_\P
\cdots
[y_n^{\pm};\alpha_n,0]_\P$
are Poisson--Noetherian.
As ordinary algebras, these are polynomial and Laurent
polynomial algebras over $A$, respectively. They are not
Noetherian. The polynomial case follows since quotienting by
$(y_1,\ldots,y_n)$ gives $A$
(which is not Noetherian), 
and the Laurent case follows
because the strictly ascending chain
$(x_1)A \subsetneq (x_1,x_2)A \subsetneq (x_1,x_2,x_3)A
\subsetneq \cdots$
extends to a strictly ascending chain of ideals in the Laurent
polynomial algebra.
\end{ex}

\section*{Acknowledgments}
The authors thank Samuel A. Lopes for a preliminary discussion and the anonymous referee for helpful comments and suggestions.

\end{document}